\documentclass[11pt]{article}

\usepackage{amsmath,amssymb,amsthm,setspace,graphicx,array,tikz,url,blkarray,easyReview,bm,hyperref}
\usepackage[text={6.5in,8.4in},top=1.3in]{geometry}

\newtheorem{theorem}{Theorem}[section]
\newtheorem{proposition}[theorem]{Proposition}
\newtheorem{lemma}[theorem]{Lemma}
\newtheorem{corollary}[theorem]{Corollary}

\title{The maximum domination number of $k$-majority tournaments}
\author{Jiangdong Ai\thanks{School of Mathematical Sciences and LPMC,
    Nankai University. {\tt jd@nankai.edu.cn}.
    Funded by the National Natural Science Foundation of China
    (No.\,12522117, No.\,12401456),
     the Fundamental and Interdisciplinary Disciplines Breakthrough Plan
    of the Ministry of Education of China (JYB2025XDXM207),
    and the Fundamental Research Funds for the Central Universities,
    Nankai University.}
     \hspace{2mm}
     Xiangjie Yi\thanks{School of Mathematical Sciences and LPMC,
    Nankai University. {\tt 2310184@mail.nankai.edu.cn}.}}
\date{}

\begin{document}

\maketitle

\begin{abstract}
A tournament is a \emph{$k$-majority tournament} if it is the pairwise majority relation of an election with $2k-1$ voters: there are $2k-1$
linear orders of the vertex set, repetitions allowed, such that $u\to v$ exactly when $u$ precedes $v$ in at least $k$ of them. Let $F(k)$ be the
maximum domination number of a $k$-majority tournament. In 2006, Alon, Brightwell, Kierstead, Kostochka and Winkler proved that $F(k)$ is finite and that $c_1k/\log k\le F(k)\le c_2k\log k$ for absolute constants
$c_1,c_2>0$, the order of growth of $F(k)$ has been open since. We prove that $F(k)=\Theta(k\log k)$. As a consequence, the maximum domination number of a tournament realized by fair $m$-sided dice is $\Theta(m\log m)$, closing the corresponding gap in the bounds of Alon et al.


\end{abstract}

\section{Introduction}

All tournaments in this paper are finite and nonempty. We write $u\to v$ for an arc from $u$ to $v$. A subset $D$ of $V(T)$
\emph{dominates} $T$ if every vertex outside $D$ receives an arc from some vertex in $D$. The minimum size of a dominating set is the \emph{domination number} $\gamma(T)$.

In an election with $2k-1$ voters, the pairwise majority relation is a tournament, with an arc from $u$ to $v$ whenever at least $k$ voters rank $u$ ahead of $v$. A tournament that can be represented by such a list of $2k-1$ linear orders, which may include repetitions,
is called a \emph{$k$-majority tournament}. McGarvey~\cite{McGarvey}
showed that every tournament has a majority representation. If the number of voters is fixed instead, how large can the domination number
be?

Alon, Brightwell, Kierstead, Kostochka and Winkler~\cite{ABKKW}
introduced the function
\[F(k)=\max\{\gamma(T): T\text{ is a }k\text{-majority tournament}\}.
\]
They proved a conjecture of Kierstead and Trotter by showing that $F(k)$ is finite for each $k$. They also determined $F(2)=3$ and proved that for $k\ge 2$,
\[c_1\frac{k}{\log k}\le F(k)\le c_2 k\log k
\]
for some absolute constants $c_1,c_2>0$. Fidler~\cite{Fidler} sharpened
the upper bound for small values of $k$. The order of magnitude of $F(k)$ as $k$ grows is given by the following theorem.

\begin{theorem}\label{thm:main}
There are absolute constants $c,C>0$ such that
\[ck\log k\le F(k)\le Ck\log k\] for every integer $k\ge 2$.
\end{theorem}

A tournament has \emph{property $S_m$} if for every set $D$ of $m$ vertices there is a vertex outside $D$ sending an arc to each vertex of $D$. This property was studied by Erd\H{o}s~\cite{Erdos} in
connection with a problem of Sch\"utte. We obtain the lower bound from the following construction.

\begin{theorem}\label{thm:grid}
Let $c_0=10^{-7}$. For every sufficiently large integer $t$, there is a $(2t+1)$-majority tournament $T_t$ in which every set $D$ of at most
$\lfloor c_0t\log t\rfloor$ vertices has a common in-neighbour outside $D$. In particular, $\gamma(T_t)>c_0t\log t$.
\end{theorem}

There is a similar question for dice. Assign a fair $m$-sided die to each vertex and direct an arc from $u$ to $v$ if the die of $u$ beats that of $v$ with probability greater than $1/2$. Let $G(m)$ be the supremum of the domination numbers of tournaments that can
be realized in this way. Alon et al.~\cite[Theorem~11]{ABKKW}
obtained $G(m)=\Omega(m/\log m)$ and $G(m)=O(m\log m)$, and
Theorem~\ref{thm:main} yields the matching lower bound.

\begin{corollary}\label{cor:dice}
$G(m)=\Theta(m\log m)$.
\end{corollary}

To prove Theorem~\ref{thm:grid}, we must construct, for every set
$D$ with $|D|\le c_0t\log t$, a vertex $x$ that beats every
member of $D$. The vertices in our construction are described
by $t$ coordinates. Since $D$ may contain on the order of
$t\log t$ vertices, we need to use individual coordinates for
several targets at once. We achieve this by restricting the
allowed sequences of coordinate labels using a suitable random array.

The array has two complementary properties. The first allows us to choose an allowed label sequence for $x$ that controls how many times each target needs to be served, with the sum of the squares of these requirements bounded by $O(|D|)$. The second provides groups of targets that can be served together in an unused coordinate. When many targets remain, these groups have
logarithmic size. We process the requirements in layers, each consisting of the targets whose requirements are at least a given level. Each target is served once in each of its layers, and the group sizes are adapted as each layer is depleted. Combining the bounds for the individual layers shows that the total number of coordinates used is at most $t$. The construction and its analysis are given in Sections~\ref{sec:grid}--\ref{sec:lower}.

P\'{a}lv\"{o}lgyi and Gy\'{a}rf\'{a}s~\cite{PG} studied domination in
tournaments with transitive edge colourings. They showed how majority representations give such colourings and related tournament domination
to covering point sets by boxes. Another line of work asks how large a transitive subtournament is guaranteed in an $n$-vertex $k$-majority tournament. In our notation, the general lower bound of Milans, Schreiber and West~\cite[Theorem~3.1]{MSW} is
$n^{1/3^{k-1}}$, which Shapira and Yuster~\cite{SY} improved to $n^{\Omega(1/k)}$.

One may also allow an arc from $u$ to $v$ whenever the fraction of voters ranking $u$ ahead of $v$ is at least $1/2-\varepsilon$. With no restriction on the number of voters, the worst-case domination number is of order $\varepsilon^{-2}$: the upper bound is due independently to Bourneuf, Charbit and
Thomass\'{e}~\cite{BCT} and to Charikar, Ramakrishnan and
Wang~\cite{CRWupper}, who subsequently gave a matching lower bound~\cite{CRWlower}.

The threshold bounds give only partial information about $F(k)$.
With $r=2k-1$ voters the losing side receives at most
a fraction $1/2-1/(2r)$ of the votes. The threshold
$1/2-1/(4r)$ therefore gives exactly the strict-majority tournament, so the approximate upper bound in this case gives $F(k)=O(k^2)$.
Conversely, for any $r$, the relaxed digraph contains a spanning majority tournament realized by at most $r+1$ orders: when $r$ is even, an extra order breaks ties without changing any strict majority, and both directions of a tied pair were already present.
The bound of Alon et al.~\cite{ABKKW} now shows that a profile attaining the approximate lower bound must have
$r\log(r+1)=\Omega(\varepsilon^{-2})$, so in particular
$r=\omega(1/\varepsilon)$.

Section~\ref{sec:grid} describes the grid construction, and Sections~\ref{sec:array} and~\ref{sec:lower} give the array lemma and the proof of Theorem~\ref{thm:grid}. We derive Theorem~\ref{thm:main} and Corollary~\ref{cor:dice} in
Section~\ref{sec:mainproof}. Section~\ref{sec:conclusion} gives a domination bound for the complete grid and discusses a further question.

\section{Grid tournaments}\label{sec:grid}

We write $[n]=\{1,\ldots,n\}$ for a positive integer $n$ and use natural logarithms throughout. Given linear orders
$\prec_1,\ldots,\prec_m$ on a set $V$, the \emph{margin} of $x$ over $y$ is
\[
  M(x,y)=|\{i:x\prec_i y\}|-|\{i:y\prec_i x\}| .
\]
Here $x\prec_i y$ means that $x$ precedes $y$ in the $i$th order.
When $m$ is odd, these orders give a majority tournament in which $x$ sends an arc to $y$ precisely when $M(x,y)>0$.

The construction uses $t$ copies of a $q\times q$ grid, where $t,q$ are positive integers. A \emph{local state} is a triple
$(o,\ell,\mu)\in\{\mathsf H,\mathsf V\}\times[q]\times[q]$,
where $o$ is the \emph{type}, $\ell$ the \emph{line label}, and $\mu$ the \emph{mark}. A horizontal state $(\mathsf H,\ell,\mu)$ marks column $\mu$ on row $\ell$, and a vertical state $(\mathsf V,\ell,\mu)$ marks row $\mu$ on column $\ell$. We denote the entries of a state $s$ by $o(s),\ell(s),\mu(s)$ and the type opposite to $o$ by
$\overline{o}$. A state $s$ \emph{hits} a state $s'$ when the marked point of $s$ lies on the perpendicular line of $s'$, or equivalently,
when $o(s)\ne o(s')$ and $\mu(s)=\ell(s')$.

A \emph{grid vertex} is a $t$-tuple $x=(x_1,\ldots,x_t)$ of local states. We call $\lambda(x)=(\ell(x_1),\ldots,\ell(x_t))$ the \emph{line row} of $x$. For two grid vertices $x,y$, we use the hit count $W(x,y)$ and the difference $\Phi(x,y)$ between the counts in the two directions:
\[
  W(x,y)=|\{f\in[t]:x_f\text{ hits }y_f\}|,\qquad
  \Phi(x,y)=W(x,y)-W(y,x).
\]
Let $\sigma$ be a linear order on a nonempty set $X$ of grid vertices.
In the tournament $G(X,\sigma)$, a vertex $x$ sends an arc to a distinct
vertex $y$ if $\Phi(x,y)>0$ or if $\Phi(x,y)=0$ and $x$ precedes $y$ in $\sigma$.

\begin{lemma}\label{lem:orders}
For any nonempty set $X$ of grid vertices with $t$ coordinates and labels in $[q]$, and any linear order $\sigma$ on $X$, there is a list of $4t+1$ linear orders on $X$ with margin
\[ M(x,y)=2\Phi(x,y)+s_\sigma(x,y)
\]
for all distinct $x,y\in X$, where $s_\sigma(x,y)=1$ if $x$ precedes $y$ in $\sigma$ and $s_\sigma(x,y)=-1$ otherwise. In particular, $G(X,\sigma)$ is a $(2t+1)$-majority tournament.
\end{lemma}

\begin{proof}
We use two orders for each coordinate $f$ and type $o$ to record the hits made by states of type $o$ in that coordinate. The classes
for a label $i\in[q]$ are
\[
  A_i=\{x\in X:o(x_f)=o,\ \mu(x_f)=i\},\qquad
  B_i=\{x\in X:o(x_f)=\overline{o},\ \ell(x_f)=i\}.
\]
Every vertex belongs to exactly one of these classes. For $x\in A_i$ and $y\in B_j$, the states $x_f$ and $y_f$ have different types, so
$x_f$ hits $y_f$ exactly when $i=j$.

In one order we go through the blocks $A_i\cup B_i$ in increasing order of $i$, and in the other we go through them in decreasing order. Within each block we list $A_i$ before $B_i$ in both orders.
We choose an arbitrary order within each class for the first list and reverse it for the second.

Note that the two comparisons cancel for vertices in different blocks and for vertices in the same class. The only pairs left have
one vertex $x\in A_i$ and the other $y\in B_i$. Both orders rank $x$ ahead of $y$, contributing $2$ to $M(x,y)$ and $-2$ to $M(y,x)$.

Using both types gives four orders per coordinate. A hit of $x_f$ on $y_f$ contributes $2$ to $M(x,y)$, and a hit of $y_f$ on $x_f$ contributes $-2$, so the contributions cancel when both hits occur.
Over all $t$ coordinates the margin is $2\Phi(x,y)$, and the final order $\sigma$ adds $s_\sigma(x,y)$. As $|2\Phi(x,y)|\ge2$ whenever
$\Phi(x,y)\ne0$, the extra order decides precisely the ties, as required for $G(X,\sigma)$.
\end{proof}

\section{An array with uniform occupancy}\label{sec:array}

We choose the available line rows at random. A polynomial number of rows suffices, and in this section and the next we use $q=t$ and $n=t^7$.
For $z>0$, the notation $\log_+ z$ denotes $\max\{0,\log z\}$.
Property (i) below will bound the number of coordinates in which $d_f$ can hit $x_f$ when we construct $x$ (Lemma~\ref{lem:incoming}). Property (ii) will provide groups of vertices that can receive a hit in the same coordinate
(Lemma~\ref{lem:grouping}).

\begin{lemma}\label{lem:array}
For all sufficiently large $t$, there is an array of $n$ distinct rows $y^{(1)},\ldots,y^{(n)}\in[t]^t$ with the following properties.
\begin{enumerate}
\item[(i)] For every column $f$ and label $\ell$, at most $2n/t$
rows satisfy $y_f^{(i)}=\ell$. For distinct columns $f,f'$ and
labels $\ell,\ell'$, at most $2n/t^2$ rows satisfy both
$y_f^{(i)}=\ell$ and $y_{f'}^{(i)}=\ell'$.
\item[(ii)] Let $1\le b\le a\le t\log t$ be integers with
\begin{equation}\label{eq:occupancy-condition}
  b\bigl(1+\log_+(2tb/a)\bigr)\le (\log t)/4.
\end{equation}
For every set $I\subseteq[n]$ of size $a$, at least
$\lfloor t/4\rfloor$ columns $f\in[\lfloor t/2\rfloor]$ have some label occurring at least $b$ times among the entries indexed by $I$.
\end{enumerate}
\end{lemma}

\begin{proof}
We choose the $nt$ entries independently and uniformly from $[t]$.
Any two specified rows coincide with probability $t^{-t}$, since their entries agree independently with probability $1/t$ in each column. The probability of a repeated row is therefore at most
$\binom n2t^{-t}=o(1)$.

For (i), a specified label in one column is matched by any given row with probability $1/t$, and a pair of specified labels in distinct columns is matched with probability $1/t^2$. In either case the number $C$ of matching rows is binomial with parameters
$n,p$, where $p$ is the matching probability. Since $C$ has mean $np$ and variance $np(1-p)\le np$, Chebyshev's inequality gives
\[
  \Pr(C>2np)\le \Pr(|C-np|\ge np)\le \frac{1}{np}.
\]
There are $t^2$ choices involving one column and at most $t^4$ involving two, so the probability that (i) fails is at most $t^3/n+t^6/n$. This is why $t^7$ rows suffice for (i).

For (ii), consider $a,b,I$ as in the statement and call a column good if some label occurs at least $b$ times in the rows indexed by $I$. For one column we view its entries in those rows as the first $a$ terms of an infinite sequence of independent uniform
draws from $[t]$. Denote by $E_j$ the event that every label occurs fewer than $b$ times in the first $j$ draws. The column fails with probability $\Pr(E_a)$.

The label counts become independent if we sample a Poisson number $N$ of draws, with mean $a/2$ and independent of the sequence.
Indeed, let $Z_\ell$ count the occurrences of $\ell$ in the first $N$ draws. For nonnegative integers $z_1,\ldots,z_t$ with
$s=\sum_\ell z_\ell$, we multiply the probability of $N=s$ by the conditional multinomial probability to obtain
\[
  \Pr(Z_1=z_1,\ldots,Z_t=z_t)
  =e^{-a/2}\frac{(a/2)^s}{s!}\frac{s!}{z_1!\cdots z_t!}\,t^{-s}
  =\prod_{\ell=1}^t e^{-\lambda}\frac{\lambda^{z_\ell}}{z_\ell!}.
\]
Here $\lambda=a/(2t)$, and the product shows that the counts
are independent Poisson variables with this common mean.
With $\nu=\min\{\lambda,b\}$, monotonicity of the Poisson upper tail in its mean gives
\[
  \Pr(Z_\ell\ge b)
  \ge \frac{e^{-\nu}\nu^b}{b!}
  \ge e^{-b}(\nu/b)^b
  =\exp\bigl(-b(1+\log_+(2tb/a))\bigr),
\]
where we used $\nu\le b$ and $b!\le b^b$ in the second
inequality. The condition \eqref{eq:occupancy-condition} gives
$\Pr(Z_\ell\ge b)\ge t^{-1/4}$, so independence yields
\[
\Pr(E_N)=\prod_{\ell=1}^t\Pr(Z_\ell<b)
  \le (1-t^{-1/4})^t\le e^{-t^{3/4}}.
\]
We can now return to the original sample size. Markov's inequality gives $\Pr(N\le a)\ge1/2$. On $\{N\le a\}$, the event $E_a$ implies $E_N$, and $E_a$ is independent of $N$. Consequently
$\Pr(E_a)\le2\Pr(E_N)\le 2e^{-t^{3/4}}$.

For our fixed $(a,b,I)$, the columns use disjoint sets of entries and hence fail independently. If fewer than
$\lfloor t/4\rfloor$ of the first $\lfloor t/2\rfloor$ columns are good, then at least $\lfloor t/2\rfloor-\lfloor t/4\rfloor+1\ge t/4$
of them fail. The union bound over the at most $2^{t/2}$ subsets of these columns gives failure probability at most
\[
  2^{t/2}\bigl(2e^{-t^{3/4}}\bigr)^{t/4}\le \exp(t-t^{7/4}/4).
\]
It remains to take the union over all triples. There are at most $(t\log t)^2$ choices of $(a,b)$ and at most $n^a\le n^{t\log t}$ choices of $I$ for each pair. Since $n=t^7$, this gives at most $\exp(O(t\log^2 t))$ triples, so the probability that (ii) fails is at most $\exp(O(t\log^2 t)-t^{7/4}/4)=o(1)$.

With probability $1-o(1)$, the rows are distinct and both
properties hold.
\end{proof}

\section{The lower bound}\label{sec:lower}

Let $Y$ consist of the rows of an array from Lemma~\ref{lem:array},
and let $X_Y$ consist of the grid vertices $x$ with $\lambda(x)\in Y$.
All types and marks are available for each row, so $|X_Y|=n(2t)^t$.
This restriction on the rows matters. As Proposition~\ref{prop:complete}
shows, the complete grid with $q=t$ has domination number at most $t$.
For an arbitrary linear order $\sigma$ on $X_Y$, the tournament
$T_t=G(X_Y,\sigma)$ is $(2t+1)$-majority by Lemma~\ref{lem:orders}.
We seek a vertex $x$ with $\Phi(x,d)>0$ for every member of a given
family $D\subseteq X_Y$ of size at most $c_0t\log t$.

Once the line row $y=\lambda(x)$ is chosen, we can bound the number of
coordinates at which $d_f$ could hit $x_f$. For a grid vertex $d$
and a row $y\in[t]^t$, the quantity
\[
  g_d(y)=\bigl|\{f\in[t]:\mu(d_f)=y_f\}\bigr|
\]
counts the coordinates where $d_f$ has the required mark. In
particular, $W(d,x)\le g_d(y)$ for every $x$ with line row $y$.

\begin{lemma}\label{lem:incoming}
For every $D\subseteq X_Y$, some $y\in Y$ satisfies
\[
  \sum_{d\in D}(1+g_d(y))^2\le 9|D|.
\]
\end{lemma}

\begin{proof}
Choose $y$ uniformly from $Y$, and for a fixed $d$ let $I_f$
indicate that $\mu(d_f)=y_f$. Lemma~\ref{lem:array}(i) gives
$\mathbb E I_f\le 2/t$ and $\mathbb E(I_f I_{f'})\le 2/t^2$
for distinct $f,f'$. Since $g_d(y)=\sum_f I_f$ and $I_f^2=I_f$,
\[
  \mathbb E(1+g_d(y))^2
  =1+3\sum_f\mathbb E I_f+2\sum_{f<f'}\mathbb E(I_f I_{f'})
  \le 1+6+2\binom t2\frac{2}{t^2}\le 9.
\]
The sum over $d\in D$ therefore has average at most $9|D|$, and
some row attains at most the average.
\end{proof}

The group size we can guarantee depends on how many vertices still
need a hit. With $h(u)=1+\log(t\log t/u)$ for $0<u\le t\log t$,
we use the group size
\[
  B(u)=\max\{1,\lfloor\kappa\log t/h(u)\rfloor\},
  \qquad \kappa=\frac{1}{16(1+\log 64)}.
\]
The factor $h(u)$ reflects the estimate below, where we bound
$\log_+(2tb/a)$ in \eqref{eq:occupancy-condition} by
$\log(t\log t/u)$ plus a term depending on the row multiplicities.
The choice of $\kappa$ comes from the occupancy condition
\eqref{eq:occupancy-condition}. We use groups while at least
$u_0=\lceil t^{1/3}\rceil$ vertices remain. The cutoff has two useful
properties: $(\log t)^3=o(u_0)$ for the grouping estimate below,
and $u_0\sqrt{t\log t}=o(t)$ for the remaining individual
assignments. Note that $h(u)\ge1$, so $B(u)\le u_0$ for large $t$.

\begin{lemma}\label{lem:grouping}
Suppose $U\subseteq X_Y$, $u_0\le |U|=u\le t\log t$, and
$J\subseteq[\lfloor t/2\rfloor]$ with $|J|<\lfloor t/4\rfloor$.
There is a column $f\in[\lfloor t/2\rfloor]\setminus J$ in which
at least $B(u)$ vertices of $U$ have the same type and line label.
\end{lemma}

\begin{proof}
Any available column works when $B(u)=1$. Assume then that
$B=B(u)\ge2$, so $Bh(u)\le\kappa\log t$, and let $m_i$ count the
vertices of $U$ with line row $y^{(i)}$. If some $m_i\ge2B$, the
vertices on that row already agree on their line label in every
column, and at least $B$ of them agree on their type in any given
column.

We may therefore assume $m_i<2B$ for every $i$. To apply
Lemma~\ref{lem:array}(ii), we first find many rows with comparable
multiplicities in $U$. For powers of two $r$, let
$R_r=\{i:r\le m_i<2r\}$ and let $U_r$ consist of the vertices
whose row indices lie in $R_r$. The $U_r$ partition $U$, and
since $\sum_{r:\,U_r\ne\varnothing}1/r<2$, there is a class with
$|U_r|\ge u/(2r)$. Each row accounts for fewer than $2r$ vertices
in $U_r$, giving $|R_r|>u/(4r^2)$. We will use
$a=\left\lfloor\frac{u}{4r^2}\right\rfloor$ and $b=\left\lceil\frac{2B}{r}\right\rceil$. Here $r<2B\le2\kappa\log t$, whereas $u\ge t^{1/3}$, so for
large $t$ we have $a\ge u/(8r^2)\ge2\kappa\log t$. In particular,
$1\le b\le2B\le a\le t\log t$. To check
\eqref{eq:occupancy-condition}, note that $b\le4B/r$ and $B\le\log t$, whence
\[
  \frac{2tb}{a}\le \frac{64Brt}{u}\le 64r\frac{t\log t}{u}.
\]
Using $\log r\le r-1$ and $t\log t/u\ge1$, we obtain
\begin{align*}
 b\bigl(1+\log_+(2tb/a)\bigr)
 &\le \frac{4B}{r}\bigl(1+\log(64r)+\log(t\log t/u)\bigr)\\
 &\le 4(1+\log 64)Bh(u)
 \le 4(1+\log 64)\kappa\log t = (\log t)/4 .
\end{align*}
We can now apply Lemma~\ref{lem:array}(ii) to any $a$ indices from
$R_r$. At least one of the $\lfloor t/4\rfloor$ columns supplied
by the lemma lies outside $J$. In that column, $b$ rows share a
label, accounting for at least $br\ge2B$ vertices of $U$ with
that label. At least $B$ also agree on their type.
\end{proof}

\begin{lemma}\label{lem:steps}
For an integer $1\le s\le t\log t$, start with $u_1=s$ and
continue with $u_{i+1}=u_i-B(u_i)$ while $u_i\ge u_0$. The number
of steps is at most $2\varphi(s)/(\kappa\log t)$, where
\[
  \varphi(s)=\int_0^s h(z)\,dz=s\bigl(2+\log(t\log t/s)\bigr).
\]
\end{lemma}

\begin{proof}
Each index $j\in[s]$ carries weight $h(j)$. A step from $u$ to
$u-B(u)$ uses the weights with indices $u-B(u)+1,\ldots,u$, each
of which is at least $h(u)$ because $h$ is decreasing. The weight
of a step is therefore at least
\[
  B(u)h(u)\ge \frac{\kappa\log t}{2},
\]
using $\max\{1,\lfloor w\rfloor\}\ge w/2$ with $w=\kappa\log t/h(u)$.

The indices used in different steps are disjoint. Since the integral
of $h$ over $(j-1,j]$ is at least $h(j)$, the combined weight
of any $N$ steps satisfies
\[
  N\frac{\kappa\log t}{2}
  \le \sum_{j=1}^s h(j)
  \le \int_0^s h(z)\,dz
  =\varphi(s).
\]
This gives the claimed bound for the full recurrence.
\end{proof}

\begin{proof}[Proof of Theorem~\ref{thm:grid}]
Any vertex serves when $D=\varnothing$, so assume $1\le|D|\le
c_0t\log t$. Choose $y$ as in Lemma~\ref{lem:incoming}, write
$g_d=g_d(y)$, and let $L=\max_{d\in D}g_d$.
For $0\le j\le L$, consider the layer $D_j=\{d\in D:g_d\ge j\}$
of size $s_j=|D_j|$. The squared demand bound gives
\begin{equation}\label{eq:layer-size}
  s_j\le \frac{9|D|}{(j+1)^2},\qquad
  L+1\le 3\sqrt{|D|}.
\end{equation}
We construct $x$ with line row $y$, arranging one hit on the
local states of each vertex in each of its layers. The same column
can serve several vertices, but we never reuse a column in a later
step.

We work through $D_0,\ldots,D_L$ using the first $\lfloor t/2\rfloor$
columns. If $u\ge u_0$ vertices in the current layer still need
a hit, Lemma~\ref{lem:grouping} finds an unused column $f$ where
$B(u)$ of them have a common type $\theta$ and line label $\ell$.
With $x_f=(\bar\theta,y_f,\ell)$, the state $x_f$ hits $d_f$
for each of these vertices $d$. This leaves $u-B(u)$ vertices
still to be served in that layer. Once fewer than $u_0$ remain,
we move on to the next layer and attend to these vertices later.

We can always find the next column. Indeed, the remaining sizes in
layer $D_j$ follow the recurrence of Lemma~\ref{lem:steps} with
initial value $s_j$, and $\varphi$ is increasing because $h>0$.
Writing $\alpha=9c_0$, we use \eqref{eq:layer-size} to bound the
number of columns used in any initial segment of this procedure by
\begin{align*}
  \frac{2}{\kappa\log t}\sum_{j=0}^{L}\varphi(s_j)
  &\le \frac{2\alpha t}{\kappa}
  \sum_{m\ge1}\frac{2+\log(1/\alpha)+2\log m}{m^2}\\
  &\le \frac{4\alpha}{\kappa}\bigl(5+\log(1/\alpha)\bigr)t.
\end{align*}
For the last estimate, use $\log m\le\sqrt m$ and the integral
bounds $\sum_{m\ge1}m^{-2}\le2$ and $\sum_{m\ge1}m^{-3/2}\le3$.
Here $\alpha=9\cdot10^{-7}$, so $\log(1/\alpha)<14$.
Also $1/\kappa<83$, giving a bound of less than $t/8$.
Before each step, fewer than $\lfloor t/4\rfloor$ columns have
been used, as Lemma~\ref{lem:grouping} requires.

At most $(L+1)(u_0-1)=O(t^{5/6}\sqrt{\log t})=o(t)$ pairs
$(d,j)$ with $d\in D_j$ still need a hit. The remaining
$\lceil t/2\rceil$ columns suffice for individual assignments.
Give each of these pairs its own column $f>\lfloor t/2\rfloor$,
where $x_f=(\overline{o(d_f)},y_f,\ell(d_f))$ hits $d_f$. At
columns still unused, we may take $x_f=(\mathsf H,y_f,1)$.

The resulting vertex $x$ lies in $X_Y$ because its line row is
$y$. Each $d$ belongs to $g_d+1$ layers and receives a distinct
coordinate in each, so $W(x,d)\ge g_d+1>W(d,x)$. Hence $x\notin D$
and sends an arc to every vertex of $D$, regardless of the choice
of $\sigma$.
\end{proof}

\section{Proofs of the main results}\label{sec:mainproof}

\begin{proof}[Proof of Theorem~\ref{thm:main}]
The upper bound is Alon et al.'s theorem~\cite[Theorem~4]{ABKKW}.
For the lower bound, observe that we can increase the number of orders
by two without changing any margin by appending an order and its
reverse, so $F$ is nondecreasing. With
$t=\lfloor(k-1)/2\rfloor$, Theorem~\ref{thm:grid} gives
\[
  F(k)\ge F(2t+1)\ge\gamma(T_t)>c_0t\log t
\]
for all sufficiently large $k$. Here $t\sim k/2$, so a suitable
positive constant $c$ works for large $k$. Since $F(k)\ge1$, we
may decrease $c$ to cover the remaining values.
\end{proof}

\begin{proof}[Proof of Corollary~\ref{cor:dice}]
The upper bound is due to Alon et al.~\cite[Theorem~11]{ABKKW}.
For the lower bound, consider a tournament attaining $F(k)$ for
$k=\lfloor(m+2)/4\rfloor$, where $m$ is sufficiently large.
Repeating each of its $2k-1$ orders twice gives a list of
$4k-2\le m$ orders, with every margin at least $2$ in absolute
value. We extend the list to length $m$ by appending orders in
reverse pairs and, if $m-(4k-2)$ is odd, one extra order. The
pairs do not affect the margins, and the extra order can change
a margin by only $1$. Thus all margin signs are preserved.

We use $m$ disjoint blocks of integers, with every label in a
block smaller than every label in the next block. In block $i$,
the vertices receive distinct labels, larger labels going to
vertices earlier in the $i$th order. The die for a vertex has
one face from each block, bearing the label assigned to that vertex.
There are no ties between two different dice. Comparisons
across blocks contribute equally many wins to the two dice, whereas
the $m$ comparisons within blocks have the same outcomes as the
$m$ orders. The dice therefore realize the chosen tournament.
Since $k=\Theta(m)$, Theorem~\ref{thm:main} gives
$G(m)\ge F(k)=\Omega(m\log m)$.
\end{proof}

\section{Concluding remarks}\label{sec:conclusion}

The choice of line rows matters. For $q=t$, the complete grid has
a dominating set of at most $t$ vertices, regardless of the order
used to break ties.

\begin{proposition}\label{prop:complete}
Let $X$ be the set of all grid vertices with $t$ coordinates and $q$
labels. For every linear order $\sigma$ of $X$ we have
$\gamma(G(X,\sigma))\le q$.
\end{proposition}

\begin{proof}
We construct a dominating set of $q$ vertices containing the
first vertex $v$ of $\sigma$. All of them have the same types as
$v$, and in each coordinate they use every line label and every
mark exactly once. More precisely, for each $f\in[t]$ choose
permutations $\rho_f,\eta_f$ of $[q]$ with
$\rho_f(1)=\ell(v_f)$ and $\eta_f(1)=\mu(v_f)$. The vertices
$p^{(j)}$, $j\in[q]$, are given by
\[
  p_f^{(j)}=(o(v_f),\rho_f(j),\eta_f(j))\qquad(f\in[t]),
\]
so $p^{(1)}=v$. For $x$ outside $D=\{p^{(j)}:j\in[q]\}$, consider
a coordinate $f$. If $o(x_f)=o(v_f)$, there are no hits between
$x_f$ and the $p_f^{(j)}$. Otherwise exactly one $p_f^{(j)}$ has
the mark $\ell(x_f)$ and exactly one has the line label $\mu(x_f)$,
giving one hit in each direction. Hence
$\sum_{j=1}^q\Phi(p^{(j)},x)=0$. Either one summand is positive,
in which case the corresponding $p^{(j)}$ beats $x$, or all
summands vanish and $v$ beats $x$ because it comes first in $\sigma$.
\end{proof}

For $m=\lfloor c_0t\log t\rfloor$, Theorem~\ref{thm:grid} gives
property $S_m$ on $t^7(2t)^t=2^{\Theta(m)}$ vertices. Recall that
Erd\H{o}s~\cite{Erdos} showed that a tournament with this property
has at least $2^{m+1}-1$ vertices. Our construction thus has the
right exponential order. A natural question is whether the array
in Lemma~\ref{lem:array} can be constructed explicitly.


\begin{thebibliography}{10}

\bibitem{ABKKW}
N. Alon, G. Brightwell, H. A. Kierstead, A. V. Kostochka and P. Winkler,
Dominating sets in $k$-majority tournaments,
\emph{J. Combin. Theory Ser. B} \textbf{96} (2006), 374--387.

\bibitem{BCT}
R. Bourneuf, P. Charbit and S. Thomass\'{e},
A dense neighborhood lemma: Applications of partial concept classes
to domination and chromatic number,
\emph{Proc. 2025 IEEE 66th Annual Symp. Foundations of Computer
Science (FOCS)}, IEEE, 2025, 1--37.

\bibitem{CRWupper}
M. Charikar, P. Ramakrishnan and K. Wang,
Approximately dominating sets in elections,
\emph{Proc. 2026 ACM--SIAM Symp. Discrete Algorithms (SODA)},
SIAM, 2026, 1747--1760.

\bibitem{CRWlower}
M. Charikar, P. Ramakrishnan and K. Wang,
Pessimal elections for approximately dominating sets,
\emph{Proc. 2027 SIAM Symp. Simplicity in Algorithms (SOSA)}, to appear.
\href{https://arxiv.org/abs/2608.06872}{arXiv:2608.06872}.

\bibitem{Erdos}
P. Erd\H{o}s, On a problem in graph theory,
\emph{Math. Gaz.} \textbf{47} (1963), 220--223.

\bibitem{Fidler}
D. Fidler, A recurrence for bounds on dominating sets in $k$-majority
tournaments, \emph{Electron. J. Combin.} \textbf{18} (2011), \#P166.

\bibitem{McGarvey}
D. C. McGarvey, A theorem on the construction of voting paradoxes,
\emph{Econometrica} \textbf{21} (1953), 608--610.

\bibitem{MSW}
K. G. Milans, D. H. Schreiber and D. B. West,
Acyclic sets in $k$-majority tournaments,
\emph{Electron. J. Combin.} \textbf{18} (2011), \#P122.

\bibitem{PG}
D. P\'{a}lv\"{o}lgyi and A. Gy\'{a}rf\'{a}s,
Domination in transitive colorings of tournaments,
\emph{J. Combin. Theory Ser. B} \textbf{107} (2014), 1--11.

\bibitem{SY}
A. Shapira and R. Yuster, On Ramsey properties of $k$-majority
tournaments, \href{https://arxiv.org/abs/2603.04174}{arXiv:2603.04174}, 2026.

\end{thebibliography}
\end{document}